\documentclass[12pt]{article}
\usepackage{amsmath, amsthm}
\usepackage{amssymb,amsmath,graphicx,enumerate}
\usepackage{longtable,mathrsfs}
\usepackage{hyperref}
\newtheoremstyle{theorem}
  {10pt}		  
  {10pt}  
  {\sl}  
  {\parindent}     
  {\bf}  
  {. }    
  { }    
  {}     
 \newtheorem{theorem}{Theorem}[section]
\newtheorem{lemma}{Lemma}[section]

\newtheorem{definition}{Definition}[section]
\newtheorem{example}{Example}[section]
\newtheorem{remark}{Remark}[section]
\newtheorem{proposition}{Proposition}[section]
\allowbreak
\newtheoremstyle{defi}
  {10pt}		  
  {10pt}  
  {\rm}  
  {\parindent}     
  {\bf}  
  {. }    
  { }    
  {}     
\theoremstyle{defi}

\begin{document}
 \begin{center}
\large    Hadamard type  fuzzy inequality for  $(s,m)$-convex function in second sense
\end{center}
\begin{center}
                 $^{1}$  Deepak B. Pachpatte and  $^{2}$  Kavita U. Shinde
\end{center}
\begin{center}
$^{1}$  Department of Mathematics,\\
Dr.Babasaheb Ambedkar Marathwada
University, Aurangabad-431 004 (M.S) India.\\
E-mail: pachpatte@gmail.com \\
$^{2}$ Department of Mathematics,\\
Dr.Babasaheb Ambedkar Marathwada
University, Aurangabad-431 004 (M.S) India.\\

E-mail: kansurkar14@gmail.com   \\

\end{center}

\begin{abstract}
In this paper we prove a   Hadamard type fuzzy inequality for $(s,m)$-convex function in second sense and some examples are given.
\end{abstract}

\section{Introduction}
 \quad Fuzzy measure and fuzzy integral first introduced by Sugeno \cite{0}. It
 can be used for modelling problems in non-deterministic environment. The use
 of the Sugeno integral can be envisaged from two points of view, decision under uncertainty and multi-criteria decision making \cite{dubo}.

The integral inequalities are significant mathematical tools both in theory and applications. The integral inequalities such as  Jensen, Holder, Chebyshev
and Minkowski are widely used in various fields of mathematics including forecasting of time-series, information science, decision making under risk
and probability theory, differential equations.

Hanson \cite{han} gave the notion of invexity as a significant generalization of classical convexity. Ben-Israel and Mond \cite{isra} introduced the preinvex functions a special case of invex functions. Latif and Shoaib \cite{lati} discussed the concept of $m$-preinvex functions and $(\alpha,m)$-preinvex functions.
In \cite{gill, sari, ngoc, gzban} author studied Hermite-Hadamard type inequalities for $r$-convex function.

The study of inequalities for Sugeno integral was initiated by Roman-Flores et.al. \cite{hroman}, \cite{flores}. Since  many authors have  studied the
 different types of  integral inequalities for fuzzy integral see \cite{h}-\cite{hroman}, \cite{haip, LI, HH}.

In this paper we give the Hadamard type inequality for $(s,m)$-convex functions in second sense with respect to Sugeno integral.
\section{Preliminary}
  \quad  The definitions and basic properties of fuzzy measures and fuzzy integrals that will be used in the next sections and can be found in \cite{1z}, \cite{0}.

 Suppose that $\wp$ is a $\sigma$-algebra of subsets of $X$ and $\mu:\wp\longrightarrow[0,\infty)$ be a non-negative, extended real valued set function.
We say that $\mu$ is a fuzzy measure if
\begin{enumerate}
  \item $\mu(\phi)=0$;
  \item $ E,F\in\wp$ and $E\subset F$ imply $\mu(E)\leq\mu(F)$;
  \item  $\{E_{n}\}\subset\wp, E_{1}\subset E_{2}\subset..., $ imply
   $\lim_{n\longrightarrow\infty}\mu(E_{n})=\mu(\bigcup_{n=1}^{\infty}E_{n})$;
  \item $\{E_{n}\}\subset\wp$, $E_{1}\supset E_{2}\supset...,$ $\mu(E_{1})<\infty$, imply $\lim_{n\longrightarrow\infty}\mu(E_{n})=\mu(\bigcap_{n=1}^{\infty}E_{n})$.
\end{enumerate}
\par If $f$ is  non-negative real-valued function defined on $X$, we denote the set $\{x\in X:f(x)\geq\alpha\}=\{x\in X:f\geq\alpha\}$
 by $F_{\alpha}$  for $\alpha\geq 0$, where if $\alpha \leq \beta$ then $ F_{\beta}\subset F_{\alpha}$.

 Let $(X,\wp,\mu)$ be a fuzzy measure space, we denote by $M^{+}$ the set of all non-negative measurable functions with respect to $\wp$.
 \begin{definition}(Sugeno \cite{0}).
   Let $(X,\wp,\mu)$ be a fuzzy measure space, $f\in M^{+}$ and $A\in\wp$, the Sugeno integral (or fuzzy integral)
of $f$ on $A$, with respect to the fuzzy measure $\mu$, is defined as
\begin{equation*}
   (s)\int_{A}f d\mu=\underset{\alpha\geq0}{\bigvee}[\alpha\wedge\mu(A\cap F_{\alpha})],
\end{equation*}
when $A= X$,
\begin{equation*}
  (s)\int_{X}f d\mu=\underset{\alpha\geq0}{\bigvee}[\alpha\wedge\mu(F_{\alpha})],
\end{equation*}
where $\bigvee$ and $\wedge$ denote the operations sup and inf on $[0,\infty)$, respectively.
 \end{definition}
 Some of the properties of fuzzy integrals are as follows.
 \begin{proposition}\label{z2}\cite{2z}
  Let $(X,\wp,\mu)$ be fuzzy measure space, $A,B\in \wp$ and $f,g\in M^{+}$ then:
\begin{enumerate}
  \item  $(s)\int_{A} f d \mu \leq \mu(A)$;
  \item  $(s)\int_{A} kd \mu = k\wedge\mu(A)$, $k$ for non-negative constant;
  \item $(s)\int_{A} f d \mu \leq (s)\int_{A}gd\mu$, for $f\leq g$;
  \item $\mu(A\cap\{f\geq\alpha\})\geq\alpha\Longrightarrow(s)\int_{A}fd\mu\geq\alpha$;
  \item$\mu(A\cap\{f\geq\alpha\})\leq\alpha\Longrightarrow(s)\int_{A}fd\mu\leq\alpha$;
  \item$(s)\int_{A} f d \mu >\alpha \Longleftrightarrow$ there exists $\gamma>\alpha$ such that
   $\mu(A\cap\{f\geq\gamma\})>\alpha$;
  \item $(s)\int_{A} f d \mu <\alpha \Longleftrightarrow$ there exists $\gamma<\alpha$
   $\mu(A\cap\{f\geq\gamma\})<\alpha$.
\end{enumerate}
\end{proposition}
\begin{remark}\label{z1}
   Consider the distribution function $F$ associated to $f$ on $A$, that is, $F(\alpha)=\mu(A\cap\{f\geq\alpha\})$.Then due to
$(4)$ and $(5)$ of Proposition \ref{z2},  we have $F(\alpha)=\alpha\Longrightarrow (s)\int_{A}f d\mu=\alpha$.  Fuzzy integral can be obtained by solving the equation $F(\alpha)=\alpha$.
\end{remark}
 \begin{definition} \cite{park}
  Let $(s,m)\in (0,1]^2$ be a pair of real numbers. A function $f:I \subseteq R_{+}\longrightarrow R$ is said to be $(s,m)$-convex function in second
  sense  if
   \begin{equation}\label{lq30}
     f(\lambda x+ m(1-\lambda)y)\leq \lambda^s f(x)+m(1-\lambda)^s f(y),
   \end{equation}
   holds for all $(x,y)\in I$ and $\lambda\in [0,1].$
\end{definition}
  Some inequalities for $(s,m)$-convex functions in second sense  are obtained  in \cite{park}-\cite{liy}. If $(s,m)=(1,1)$, then we  obtain the definition of convex function.
  If $(s,m)=(s,1),$ then we obtain definition of $s$-convex function in the second sense. It is denoted by $K^{2}_{s,m}$, the set of all $(s,m)$-convex functions in the second sense.\\
  Now we give the Lemma proved in \cite{haip}
 \begin{lemma}\label{lmm} \cite{haip}
    Let $x\in[0,1],$ then  the inequality
\begin{equation}\label{lq42}
      (1-x)^s\leq 2^{1-s}-x^s,
\end{equation}
holds for $s\in(0,1].$
 \end{lemma}
 \section{Main Results}
In \cite{kir} U. S. Kirmaci et.al. proved the following Hadamard type inequalities for product of convex function and $s$-convex functions.
\begin{theorem}
  Let $f,g:[a,b]\longrightarrow \mathbb{R},$ $a,b\in[0,\infty),$ $a<b$ be functions such that $g$ and $fg$ are in $L^1([a,b]).$ If $f$ is convex
  and non-negative  on $[a,b]$ and $g$ is $s$-convex function on $[a,b]$ for some fixed $s\in(0,1)$ then
  \begin{equation}\label{lq1}
    \frac{1}{b-a}\int_{a}^{b} f(x)g(x)dx \leq \frac{1}{s+2}M(a,b)+\frac{1}{(s+1)(s+2)}N(a,b),
  \end{equation}
  where, $M(a,b)=f(a)g(a)+f(b)g(b)$ and $N(a,b)=f(a)g(b)+f(b)g(a).$
  \end{theorem}
   Now consider an example.
\begin{example}
Consider $X=[0,1]$ and let $\mu$ be the Lebesgue measure on $X$. If we take the function $f(x)=\frac{x^2}{2}$ and $g(x)=\frac{x^3}{2},$
  $f(x),g(x)\in K^{2}_{s,1}$ for $s\in(0, 1/3].$ Let $s=1/3$ the Sugeno integral
  \begin{equation*}
    (s)\int_{0}^{1}\frac{x^5}{4}d\mu=0.1269.
  \end{equation*}
  Also,
  $\frac{1}{s+2}M(a,b)+\frac{1}{(s+1)(s+2)}N(a,b)=0.1071.$
  \end{example}
This proves that the right hand side  \eqref{lq1}  of Hadamard type inequalities for $(s,m)$-convex functions in second sense is not satisfied for Sugeno integral.

In this section we give an Hadamard type inequalities for  product of $(s,m)$-convex function in second sense which is based on Sugeno integral.
\begin{theorem}\label{1th}
 Let $\mu$  be the Lebesgue measure on $\mathbb{R}$. Let $(s,m)\in (0,1]^2$ and $f,g:[a,b]\longrightarrow[0,\infty)$ are $(s,m)$-convex functions in second sense, such that $f(b)>mf(a)$ and
  $g(b)>mg(a)$ then
  \begin{equation*}
    (s)\int_{a}^{b} f g d\mu\leq min\{\beta, b-a\},
  \end{equation*}
  where $\beta$ is given by
  \begin{align}\label{lq2}
      &\nonumber(b-ma)^2- (b-ma)^2 \bigg(\frac{\beta-m2^{1-s}g(a)}{g(b)-mg(a)}\bigg)^{\frac{1}{s}}-
    (b-ma)^2 \bigg(\frac{\beta-m2^{1-s}f(a)}{f(b)-mf(a)}\bigg)^{\frac{1}{s}}\\
    &+(b-ma)^2 \bigg(\frac{\beta-m2^{1-s}f(a)}{f(b)-mf(a)}\bigg)^{\frac{1}{s}}.\bigg(\frac{\beta-m2^{1-s}g(a)}{g(b)-mg(a)}\bigg)^{\frac{1}{s}}=\beta.
  \end{align}
  \end{theorem}
\begin{proof}
  Let $f(x),g(x)\in K^{2}_{s,m}$ for $x\in[a,b],$ we have
  \begin{align}\label{lq3}
    \nonumber f(x)=& f\bigg(m\bigg(1-\frac{x-ma}{b-ma}\bigg)a+\bigg(\frac{x-ma}{b-ma}\bigg)b\bigg)\\
   \leq& m \bigg(1-\frac{x-ma}{b-ma}\bigg)^s f(a)+ \bigg(\frac{x-ma}{b-ma}\bigg)f(b).
    \end{align}
\begin{align}\label{lq4}
    \nonumber g(x)=& g\bigg(m\bigg(1-\frac{x-ma}{b-ma}\bigg)a+\bigg(\frac{x-ma}{b-ma}\bigg)b\bigg)\\
     \leq& m \bigg(1-\frac{x-ma}{b-ma}\bigg)^s g(a)+ \bigg(\frac{x-ma}{b-ma}\bigg)g(b).
  \end{align}
  Form Lemma \ref{lmm}, we have
\begin{equation}\label{lq6}
    \bigg(1-\frac{x-a}{b-a}\bigg)^s\leq 2^{1-s}-\bigg(\frac{x-a}{b-a}\bigg)^s.
\end{equation}
  Thus, from \eqref{lq3}, \eqref{lq4} and \eqref{lq6},  we have
\begin{align}\label{lq5}
   \nonumber f(x)\leq& m2^{1-s}f(a)+\bigg(\frac{x-ma}{b-ma}\bigg)^s[f(b)-mf(a)]=p_{1}(x).\\
    g(x)\leq& m2^{1-s}g(a)+\bigg(\frac{x-ma}{b-ma}\bigg)^s[g(b)-mg(a)]=p_{2}(x).
\end{align}
By Proposition \ref{z2}, we have
\begin{align}\label{lq7}
  \nonumber (s)\int_{a}^{b} fg d\mu \leq& (s)\int_{a}^{b}\bigg( m2^{1-s}f(a)+\bigg(\frac{x-ma}{b-ma}\bigg)^s[f(b)-mf(a)]  \bigg).\\
  &\nonumber \qquad\qquad\bigg(  m2^{1-s}g(a)+\bigg(\frac{x-ma}{b-ma}\bigg)^s[g(b)-mg(a)]\bigg)d\mu\\
  =&  (s)\int_{a}^{b} p_{1}(x)p_{2}(x)d\mu.
\end{align}
To calculate Sugeno integral, we consider the distribution function $F$ given by
\begin{align}\label{lq9}
  \nonumber F(\beta)=&\mu([a,b]\cap\{x|p_{1}(x)p_{2}(x)\geq\beta\})\\
  \nonumber =& \mu([a,b]\cap\{x|p_{1}(x)\geq\beta\}).\mu([a,b]\cap\{x|p_{2}(x)\geq\beta\})\\
  \nonumber=& \mu\bigg([a,b]\cap\bigg\{x|m2^{1-s}f(a)+\bigg(\frac{x-ma}{b-ma}\bigg)^s[f(b)-mf(a)]\geq\beta\bigg\}\bigg).\\
  \nonumber &\mu\bigg([a,b]\cap\bigg\{x|m2^{1-s}g(a)+\bigg(\frac{x-ma}{b-ma}\bigg)^s[g(b)-mg(a)]\geq\beta\bigg\}\bigg)\\
  \nonumber =& \mu\bigg([a,b]\cap\bigg\{x|x\geq (b-ma)\bigg(\frac{\beta-m2^{1-s}f(a)}{f(b)-mf(a)}\bigg)^{\frac{1}{s}}+ma\bigg\}\bigg).\\
  \nonumber &\mu\bigg([a,b]\cap\bigg\{x|x\geq (b-ma)\bigg(\frac{\beta-m2^{1-s}g(a)}{g(b)-mg(a)}\bigg)^{\frac{1}{s}}+ma\bigg\}\bigg).\\
 \nonumber =&\bigg(b-ma)-(b-ma)\bigg(\frac{\beta-m2^{1-s}f(a)}{f(b)-mf(a)}\bigg)^{\frac{1}{s}}\bigg).\\
  &\qquad\qquad\bigg((b-ma)-(b-ma)\bigg(\frac{\beta-m2^{1-s}g(a)}{g(b)-mg(a)}\bigg)^{\frac{1}{s}}\bigg)
\end{align}
Let $F(\beta)=\beta$ and solution of \eqref{lq9} is given by \eqref{lq2}. By Proposition \ref{z2} and Remark \ref{z1},  we have
 \begin{equation*}
    (s)\int_{a}^{b} f g d\mu\leq min\{\beta, b-a\}.
  \end{equation*}
\end{proof}
\begin{remark}
  Let $s=m=1,$ $(s,m)\in(0,1]^2$ and $f,g:[a,b]\longrightarrow[0,\infty)$ are convex functions such that $f(b)>f(a)$ and $g(b)>g(a).$ Let $\mu$
  be the Lebesgue measure on $\mathbb{R}.$ Then
  \begin{equation*}
    (s)\int_{a}^{b} f g d\mu\leq min\{\beta, b-a\},
  \end{equation*}
  where $\beta$ is given as
  \begin{equation}\label{lq10}
    \bigg((b-a)\bigg(1-\frac{\beta-f(a)}{f(b)-f(a)}\bigg)\bigg).  \bigg((b-a)\bigg(1-\frac{\beta-g(a)}{g(b)-g(a)}\bigg)\bigg)=\beta.
  \end{equation}
\end{remark}
\begin{theorem}\label{thm2}
  Let $\mu$  be the Lebesgue measure on $\mathbb{R}$. Let $(s,m)\in (0,1]^2$ and $f,g:[a,b]\longrightarrow[0,\infty)$ are $(s,m)$-convex functions in second sense, such that $f(b)<mf(a)$ and
  $g(b)< mg(a)$ then
  \begin{equation*}
    (s)\int_{a}^{b} f g d\mu\leq min\{\beta, b-a\},
  \end{equation*}
  where $\beta$ is given by
\begin{align}\label{lq11}
    &\nonumber(b-ma)^2\bigg(\frac{\beta-m2^{1-s}f(a)}{f(b)-mf(a)}\bigg)^\frac{1}{s}.\bigg(\frac{\beta-m2^{1-s}g(a)}{g(b)-mg(a)}\bigg)^\frac{1}{s}\\
   & \nonumber+ (b-ma)(ma-a) \bigg(\frac{\beta-m2^{1-s}f(a)}{f(b)-mf(a)}\bigg)^{\frac{1}{s}}\\
  & +(b-ma)(ma-a)\bigg(\frac{\beta-m2^{1-s}g(a)}{g(b)-mg(a)}\bigg)^\frac{1}{s}+(ma-a)^2=\beta.
  \end{align}
\end{theorem}
 \begin{proof}
   Similar to the  Theorem \eqref{1th}, consider the functions
 \begin{align}\label{lq15}
      p_{1}(x)=& m2^{1-s}f(a)+\bigg(\frac{x-ma}{b-ma}\bigg)^s[f(b)-mf(a)],
 \end{align}
 \begin{equation}\label{lq16}
     p_{2}(x) =  m2^{1-s}g(a)+\bigg(\frac{x-ma}{b-ma}\bigg)^s[g(b)-mg(a)].
 \end{equation}
 Now consider the distribution function $F$ given as
\begin{align}\label{lq17}
  \nonumber F(\beta)=&\mu([a,b]\cap\{x|p_{1}(x)p_{2}(x)\geq\beta\})\\
  \nonumber =& \mu([a,b]\cap\{x|p_{1}(x)\geq\beta\}).\mu([a,b]\cap\{x|p_{2}(x)\geq\beta\})\\
  \nonumber=& \mu\bigg([a,b]\cap\bigg\{x|m2^{1-s}f(a)+\bigg(\frac{x-ma}{b-ma}\bigg)^s[f(b)-mf(a)]\geq\beta\bigg\}\bigg).\\
  \nonumber &\mu\bigg([a,b]\cap\bigg\{x|m2^{1-s}g(a)+\bigg(\frac{x-ma}{b-ma}\bigg)^s[g(b)-mg(a)]\geq\beta\bigg\}\bigg)\\
  \nonumber =& \mu\bigg([a,b]\cap\bigg\{x|x\leq (b-ma)\bigg(\frac{\beta-m2^{1-s}f(a)}{f(b)-mf(a)}\bigg)^{\frac{1}{s}}+ma\bigg\}\bigg).\\
  \nonumber &\mu\bigg([a,b]\cap\bigg\{x|x\leq (b-ma)\bigg(\frac{\beta-m2^{1-s}g(a)}{g(b)-mg(a)}\bigg)^{\frac{1}{s}}+ma\bigg\}\bigg).\\
  \nonumber=&\bigg( (b-ma)\bigg(\frac{\beta-m2^{1-s}f(a)}{f(b)-mf(a)}\bigg)^{\frac{1}{s}}+ma-a\bigg).\\
 &\qquad\quad \bigg((b-ma)\bigg(\frac{\beta-m2^{1-s}g(a)}{g(b)-mg(a)}\bigg)^{\frac{1}{s}}+ma-a\bigg).
\end{align}
Let $F(\beta)=\beta$ and solution of \eqref{lq17} is given by \eqref{lq11}. By Proposition \ref{z2} and Remark \ref{z1}, we have
\begin{equation*}
    (s)\int_{a}^{b} f g d\mu\leq min\{\beta, b-a\}.
  \end{equation*}
 \end{proof}
\begin{remark}
  Let $\mu$ be the Lebesgue measure on $\mathbb{R}$. Let $s=1,$ $m=1$ and $(s,m)\in(0,1]^2,$ $f,g:[a,b]\longrightarrow[0,\infty)$
  are convex functions such that $f(b)<f(a)$ and $g(b)<g(a).$ Then
  \begin{equation*}
    (s)\int_{a}^{b} f g d\mu\leq min\{\beta, b-a\},
  \end{equation*}
  where $\beta$ is given as
  \begin{equation}\label{lq20}
    (b-a)^2\bigg(\frac{\beta-f(a)}{f(b)-f(a)}\bigg).\bigg(\frac{\beta-g(a)}{g(b)-g(a)}\bigg)=\beta.
  \end{equation}
\end{remark}
\begin{remark}
  If $f(b)=mf(a)$ and $g(b)=mg(a)$, from \eqref{lq15} and \eqref{lq16}, we have  $p_{1}(x)=m2^{1-s}$ and $p_{2}(x)=m2^{1-s}$  and Proposition
  \ref{z2}, we have
  \begin{equation}\label{lq22}
    (s)\int_{a}^{b} fg d\mu \leq \{m^2 2^{2-2s}f(a)g(a), b-a\}.
  \end{equation}
  \end{remark}
  \begin{example}
  Consider the function $f(x)=x^{3/2}$ and $g(x)=x^{1/2}$ then $f(x)$, $g(x)$ are convex functions i.e. $f(x),g(x)\in K^{2}_{s,m}$, where $s=m=1$.
  Let $\mu$ be the Lebesgue measure on $x=[1,4]$. Thus $f(4)>mf(1)$ and $g(4)>mg(1)$. By Theorem \ref{1th}, we have
  \begin{equation}\label{ot5}
    2.4384= (s)\int_{1}^{4}x^2 d\mu \leq min\{2.5302, 4-1\}=2.5302.
  \end{equation}
\end{example}

\begin{example}
    Consider the function $f(x)=\frac{1}{x^2}$ and $g(x)=\frac{1}{x^2}$, then $f(x)$, $g(x)$ are convex functions i.e. $f(x),g(x)\in K^{2}_{s,m}$,
    where $s=m=1$. Let $\mu$ be the Lebesgue measure on $X=[1,2]$. Thus $f(2)<mf(1)$ and $g(2)<mg(1)$. By Theorem \ref{thm2}, we have
\begin{equation}\label{ot21}
      0.3247=(s)\int_{1}^{2} \frac{1}{x^4} d\mu \leq  min\{0.4802,2-1\}=0.4802.
\end{equation}
  \end{example}

\end{document}